\documentclass{gtpart}     
\usepackage{pinlabel}  
\usepackage{graphicx}  
\usepackage{amscd}
\usepackage{amsfonts}
\usepackage{epsfig}
\usepackage{enumerate}
\usepackage{hyperref}

\hypersetup{
    pdftitle={On the vanishing of the Rokhlin invariant},
    pdfauthor={Tetsuhiro Moriyama},
	pdfsubject={57M27 (primary); 57N70, 57R20, 55R80 (secondary)},
	pdfkeywords={Rokhlin invariant, Casson invariant, integral homology 3-sphere, signature of
	manifolds, cobordism, Kontsevich-Kuperberg-Thurston invariant, configuration space},
}

\title{On the vanishing of the Rokhlin invariant}

\author{Tetsuhiro Moriyama}
\givenname{Tetsuhiro}
\surname{Moriyama}
\address{Institut Fourier\\100 rue des Maths\\BP 74 38402 St Martin d'H\'eres\\France}
\email{moriyama@fourier.ujf-grenoble.fr}
\email{tetsuhir@ms.u-tokyo.ac.jp}
\urladdr{}

\keyword{Rokhlin invariant}
\keyword{Casson invariant}
\keyword{integral homology 3-sphere}
\keyword{signature of manifolds}
\keyword{cobordism}
\keyword{Kontsevich-Kuperberg-Thurston invariant}
\keyword{configuration space}

\subject{primary}{msc2000}{57M27}   
\subject{secondary}{msc2000}{57N70} 
\subject{secondary}{msc2000}{57R20} 
\subject{secondary}{msc2000}{55R80} 

\arxivreference{}
\arxivpassword{}

\volumenumber{}
\issuenumber{}
\publicationyear{}
\papernumber{}
\startpage{}
\endpage{}
\doi{}
\MR{}
\Zbl{}
\received{}
\revised{}
\accepted{}
\published{}
\publishedonline{}
\proposed{}
\seconded{}
\corresponding{}
\editor{}
\version{}

\theoremstyle{plain}
\newtheorem{theorem}{Theorem}[section]    
\newtheorem{lemma}{Lemma}[section]          
\newtheorem{proposition}{Proposition}[section]          

\theoremstyle{definition}
\newtheorem{definition}{Definition}[section]    
\newtheorem{remark}{Remark}[section]             

\newtheorem{notation}{Notation}[section]

\makeatletter
\let\c@lemma=\c@theorem
\let\c@proposition=\c@theorem
\let\c@corollary=\c@theorem
\let\c@definition=\c@theorem
\let\c@remark=\c@theorem
\let\c@example=\c@theorem
\let\c@notation=\c@theorem
\let\c@assertion=\c@theorem
\makeatother

\theoremstyle{plain} 
\newtheorem{mainthm}{Theorem}
\newtheorem{maincor}{Corollary}          

\makeautorefname{mainthm}{Theorem}
\makeautorefname{maincor}{Corollary}

\makeatletter
\let\c@maincor=\c@mainthm
\makeatother

\numberwithin{equation}{section}

\renewcommand{\a}{\alpha}
\renewcommand{\b}{\beta}

\renewcommand{\d}{\delta}
\newcommand{\vp}{\varphi}
\renewcommand{\i}{\iota}

\newcommand{\s}{\sigma}
\newcommand{\up}{\upsilon}

\renewcommand{\l}{\lambda}

\renewcommand{\D}{\Delta}

\renewcommand{\O}{\Omega}

\renewcommand{\H}{\mathbb{H}}
\newcommand{\CP}{{\mathbb{C}P}}
\newcommand{\RP}{{\mathbb{R}P}}
\newcommand{\HP}{{\mathbb{H}P}}
\newcommand{\HX}{\Hat{X}}
\newcommand{\HM}{\Hat{M}}
\newcommand{\HV}{\Hat{V}}
\newcommand{\HQ}{\Hat{Q}}

\newcommand{\norm}[1]{{\|#1\|}}
\newcommand{\x}{\times}
\newcommand{\bd}[1]{{\partial #1}}
\newcommand{\es}{\emptyset}

\newcommand{\Image}{\operatorname{Im}}

\newcommand{\Ker}{\operatorname{Ker}}

\newcommand{\Sign}{\operatorname{Sign}}
\newcommand{\Os}{\Omega^{spin}}
\newcommand{\Oes}{\Omega_{6}^{e,spin}}

\newcommand{\Te}{\Tilde{e}}

\newcommand{\Int}{{\operatorname{Int}\,}}
\newcommand{\Aut}{{\operatorname{Aut}}}

\begin{document}

\begin{abstract}    
  It is a natural consequence of fundamental properties of the Casson invariant
  that the Rokhlin invariant $\mu(M)$ of an amphichiral integral homology $3$--sphere
  $M$ vanishes.
  In this paper, we give a new direct proof of this vanishing property.
  For such an $M$,
  we construct a manifold pair $(Y,Q)$ of dimensions $6$ and $3$
  equipped with some additional structure 
  ($6$--dimensional spin $e$-manifold),
   such that 
  $Q \cong M \amalg M \amalg (-M)$, and $(Y,Q) \cong (-Y,-Q)$.
  We prove that $(Y,Q)$ bounds a $7$--dimensional spin $e$--manifold $(Z,X)$
  by studying the cobordism group of $6$--dimensional spin  $e$-manifolds and
  the $\Z/2$--actions on the two--point configuration space of
  $M \setminus \left\{ pt \right\}$.
  For any such $(Z,X)$, the signature of $X$ vanishes, and this implies $\mu(M) = 0$.
  The idea of the construction of $(Y,Q)$ comes from the definition of the Kontsevich--Kuperberg--Thurston
  invariant for rational homology $3$--spheres.
\end{abstract}

\maketitle




\section{Introduction and Main results}
\label{sec:introduction}

\subsection{Introduction}
\label{sec:intro}
The Rokhlin invariant $\mu(M)$  
of a closed oriented spin $3$--manifold $M$ is defined by
\begin{equation*}
  \mu(M) = \Sign X \pmod{16},
\end{equation*}
where $X$ is a smooth compact oriented spin $4$--manifold bounded by $M$ as a spin manifold, and $\Sign X$ is
the signature of $X$. 
If $M$ is a $\Z/2$--homology $3$--sphere, then it admits a unique spin structure,
and so $\mu(M)$ is a topological invariant of $M$.
In 1980's, Casson defined an integer--valued invariant $\lambda(M)$, what is now called the
Casson invariant, for oriented integral homology $3$--spheres, and proved the following fundamental
properties for $\l$ (see~\cite{akbulut-casson}):
\begin{align}
  \l(-M) &= -\l(M)\label{eq:l1}\\
  8 \l(M) &\equiv \mu(M) \pmod{16}\label{eq:l2}
\end{align}
It is a natural consequence of \eqref{eq:l1} and \eqref{eq:l2} that, if 
$M$ is amphichiral (namely, $M$ admits a
self--homeomorphism reversing the orientation), 
then its Rokhlin invariant vanishes:
\begin{equation}
 M \cong -M \quad \Longrightarrow \quad \mu(M) = 0
 \label{eq:vanish}
 \end{equation}

In this paper, we give a new proof of this vanishing property 
for integral homology $3$--spheres (\fullref{maincor}).
We might say that our approach is more direct 
in the sense that we only consider the signature of $4$--manifolds or related characteristic
classes (\fullref{rem:direct}).

\begin{remark}
  Walker~\cite{walker} extended the Casson invariant to a rational--valued invariant $\l_{W}(M)$
  for oriented 
  rational homology $3$--spheres, such that $\l_{W}(M) = 2\l(M)$ if $M$ is an integral homology
  $3$--sphere.
  He proved that 
  $
	\l_{W}(-M) = -\l_{W}(M)
  $
  holds for any $M$, and 
  $
	4 |H_{1}(M;\Z)|^{2} \l_{W}(M) \equiv \mu(M) \pmod{16} 
  $
holds for any $\Z/2$--homology $3$--spheres, where
$|A|$ denotes the number of elements in a set $A$.
These two properties imply that the same statement \eqref{eq:vanish}  holds
for all $\Z/2$--homology $3$--spheres.
\end{remark}

\begin{remark}
  Some partial proofs of the vanishing property 
  have been given by several authors
  (Galewski~\cite{galewski-orientation-reversing}, Kawauchi~\cite{kawa:ori-rev}~\cite{kawa:vanish}, 
  Pao--Hsiang~\cite{pao-hsiang}, Siebenman~\cite{siebenmann-vanishing},
  etc.) before the Casson invariant was defined. 
\end{remark}

\subsection{Outline of the proof}
\label{sec:outline-proof}
We outline our proof of \eqref{eq:vanish} for integral homology
$3$--spheres (\fullref{maincor}), 
without giving precise definitions and computations.
See \fullref{sec:ee} and \fullref{sec:results} for more details.
Yet another proof is also given in \fullref{sec:yap} (see also \fullref{rem:yap}).

\rk{An invariant $\s$}
An $n$--dimensional $e$-manifold $\a = (W,V,e)$ is roughly a manifold pair $(W,V)$ of
dimensions $n$ and $n-3$ equipped with a cohomology class $e \in H^{2}(W \setminus V; \Q)$
called an $e$--class.
In our previous paper~\cite{moriyama:emb63}, we defined a rational--valued invariant
$\s(\a)$ for $6$--dimensional closed $e$-manifolds
such that $\s(-\a) = -\s(\a)$, and that
$
  \s(\bd{\b}) = \Sign X
$
for a $7$--dimensional $e$-manifold $\b = (Z,X,e)$
(\fullref{thm:sigma}).
\begin{proof}[Outline of the proof]
For an oriented integral homology $3$--sphere $M$, we construct a $6$--dimensional closed 
spin $e$-manifold 
$
	\a_{M} = (Y,Q,e_{M})
$
($Y$ and $Q$ are spin)
such that $Q \cong M \amalg M \amalg (-M)$ and $\a_{-M} \cong -\a_{M}$.
We can prove that $\a_{M}$ is spin null--cobordant
(\fullref{main:spin-cob}).
Namely, there exists a spin $e$-manifold
$\b = (Z,X,e)$ such that $\bd{\b} \cong \a_{M}$.
Therefore, 
\[
	\s(\a_{M}) = \Sign X  \equiv \mu(M) \pmod{16}.
\]
If $M \cong -M$,
then $\a_{M} \cong -\a_{M}$ and $\s(\a_{M}) = 0$.
Consequently, $\mu(M) \equiv 0$. 
\end{proof}

\subsection{\texorpdfstring{$e$-classes and $e$-manifolds}{e-classes and e-manifolds}}
\label{sec:ee}
In~\cite{moriyama:emb63}, we introduced
the notion of $e$-class and $e$-manifold.
Let $(Z,X)$ be a pair of (smooth, oriented, and compact) manifold $Z$ and a proper submanifold $X$
 ($\bd{X} \subset \bd{Z}$ and $X$ is transverse to $\bd{Z}$) of codimension $3$.
Let $\rho_{X} \co S(\nu_{X}) \to X$ be the unit sphere bundle associated with the normal bundle
$\nu_{X}$ of $X$ (identified with a tubular neighborhood of $X$), and
$e(F_{X}) \in H^{2}(S(\nu_{X}); \Z)$ the Euler class of the vertical tangent subbundle
$F_{X} \subset TS(\nu_{X})$ of $S(\nu_{X})$ with respect to $\rho_{X}$.
\begin{definition}[\cite{moriyama:emb63}]
  A cohomology class $e \in H^{2}(Z \setminus X ; \Q)$ is called an \emph{$e$-class} of
  $(Z,X)$ if
  $
	e|_{S(\nu_{X})} = e(F_{X})
  $
  over $\Q$.
  The triple $\b = (Z,X,e)$ is called an \emph{$e$-manifold}.
  Set $\dim \b  = \dim Z$.
\end{definition}
  A spin structure of $\b$ will mean a pair of spin structures of $Z$ and $X$.
We call $\b$ a \emph{spin} $e$-manifold if it has a spin structure.
The boundary of $\b$ is defined as $\bd{\b} = (\bd{Z}, \bd{X}, e|_{\bd{Z} \setminus \bd{X}})$,
and the disjoint union of two $e$-manifolds $\b_{i} = (Z_{i}, X_{i}, e_{i})$ ($i=1,2$)
is defined as $\b_{1} \amalg \b_{2} = 
 (Z_{1} \amalg Z_{2}, X_{1} \amalg X_{2}, e_{3})$, where $e_{3}$ is the $e$-class such 
that $e_{3}|_{Z_{i} \setminus X_{i}} = e_{i}$.
We also define $-\b = (-Z,-X,e)$.
We say $\b$ is \emph{closed} if $\bd{\b}$ is the \emph{empty} $e$-manifold $\es = (\es,\es,0)$.
If there exists an isomorphism
$f \co (Z_1, X_{1}) \to (Z_{2}, X_{2})$ of pair of manifolds such that $f^*e_{2} = e_{1}$, then
we say $\b_{1}$ and $\b_{2}$ are \emph{isomorphic} (denoted by $\b_{1} \cong \b_{2}$).
See~\cite[Section~2]{moriyama:emb63} for more details.

In \cite{moriyama:emb63}, we defined the following invariant $\s$ for $6$--dimensional closed
$e$-manifolds.
\begin{theorem}[\cite{moriyama:emb63}]\label{thm:sigma}
  There exists a unique rational--valued invariant $\s(\a)$ for $6$--dimensional closed
  $e$-manifolds $\a$ satisfying the following properties\textup{:}
  \begin{enumerate}[\textup{(}\upshape a\/\itshape\textup{)}]
	  \item\label{axiom1}
		$\s(-\a) = -\s(\a)$, \ $\s(\a \amalg \a') = \s(\a) + \s(\a')$.
	  \item\label{axiom2} For a $7$--dimensional $e$-manifold  $\b = (Z,X,e)$,
		$\s(\bd{\b}) = \Sign X$.
  \end{enumerate}
\end{theorem}
This invariant $\s$ is a generalization of Haefliger's invariant~\cite{Haefliger-knot} for
smooth $3$--knots in $S^{6}$ \cite[Theorem 5]{moriyama:emb63}. 

\subsection{Main results}
\label{sec:results}
If a closed spin $e$-manifold $\a$ bounds, namely, 
if there exists a spin $e$-manifold $\b$ such that $\bd{\b} \cong \a$ as a spin $e$-manifold, then
we say $\a$ is \emph{spin null--cobordant}.
We define $\O_{6}^{e,spin}$ to be the cobordism group of $6$--dimensional spin $e$-manifolds,
namely, it is an abelian group consisting of the spin cobordism classes $[\a]$ 
of $6$--dimensional closed spin $e$-manifolds $\a$, with the group structure given by the disjoint sum.  

In \fullref{sec:construction}, for an oriented integral homology $3$--sphere $M$, we construct a $6$--dimensional
closed spin $e$-manifold $\a_{M} = (Y,Q,e_{M})$ such that
$Q \cong M \amalg M \amalg (-M)$.
The following theorem will be used to 
prove the vanishing of the spin cobordism class $[\a_{M}] \in \O_{6}^{e,spin}$ of $\a_{M}$.
\begin{mainthm}\label{main:O}
  There is a unique isomorphism $\Phi \co \Oes \to (\Q/16\Z)\oplus (\Q/4\Z)$ such that
  \begin{equation}
	\Phi([W,\es,e]) \equiv 
	\left( \frac{1}{6}\int_{W}^{}p_{1}(TW) e - e^{3},\ \frac{1}{2}\int_{W}^{}e^{3} \right)
	\mod{16\Z \oplus 4\Z}
	\label{eq:Phi2}
  \end{equation}
  for any closed spin $6$--manifold $W$ and $e \in H^{2}(W;\Q)$.  
\end{mainthm}
Here, $p_{1}(TW)$ is the first Pontryagin class of the tangent bundle $TW$ of $W$.
Any element in $\Oes$ is represented by a closed spin $e$-manifold of the form $(W,\es,e)$
(\fullref{prop:pi-onto}), and that is why $\Phi$ is uniquely determined by \eqref{eq:Phi2}. 
\begin{mainthm}
  \label{main:spin-cob}
  For an oriented integral homology $3$--sphere $M$, the $6$--dimensional closed spin $e$-manifold
  $\a_{M}$ satisfies the following properties.
  \begin{enumerate}
	\item\label{item:reverse} $\a_{-M} \cong -\a_{M}$.
	\item\label{item:spin-null} $[\a_{M}] = 0$ in $\O_{6}^{e,spin}$.
  \end{enumerate}
\end{mainthm}
As a corollary of \fullref{thm:sigma} and \fullref{main:spin-cob},
we obtain a new proof of the vanishing property \eqref{eq:vanish} of the Rokhlin invariant for integral homology
$3$--spheres. 
\begin{maincor}[\cite{akbulut-casson}, \cite{walker} for $\Z/2$--homology $3$--spheres]\label{maincor}
  If an oriented integral homology $3$--sphere $M$ is amphichiral, then
  $\mu(M) = 0$.
\end{maincor}
\begin{proof}
  Assume $M \cong -M$.
  \fullref{main:spin-cob}~\eqref{item:reverse} and \fullref{thm:sigma}~\eqref{axiom1}
  implies $\s(\a_{M}) = 0$.
  By \fullref{main:spin-cob}~\eqref{item:spin-null}, there exists a $7$--dimensional spin
  $e$-manifold $\b = (Z,X,e)$ such that $\bd{\b} \cong \a_{M}$, and in particular, 
  we have $\s(\a_{M}) = \Sign X$ by \fullref{thm:sigma}~\eqref{axiom2}.
  The manifold $X$ is spin and $\bd{X} \cong Q$.
  Let us write $Q = M_{1} \cup M_{2} \cup (-M_{3})$, $M_{i} \cong M$.  
  Gluing the boundary components $M_{2}$ and $M_{3}$ of $X$ by a diffeomorphism, 
  we obtain a compact oriented spin $4$--manifold
  $X'$ such that $\bd{X'} = M_{1} \cong M$ and $\Sign X' = \Sign X$.  
  By the definition of the Rokhlin invariant, we have
  \[
  \mu(M)\hspace{-.4cm}
  \underset{\pmod{16}}{\equiv}\hspace{-.4cm}
  \Sign X' = \Sign X = \s(\a_{M}) =  0 
  \proved
  \]
\end{proof}
\begin{remark}
  \label{rem:yap}
  Yet another direct proof of \fullref{maincor} is given in \fullref{sec:yap},
  this is a shortcut to \fullref{maincor} without using \fullref{thm:sigma}.
  It follows from the properties of $\s$ that,
  if a $7$--dimensional $e$--manifold $\b = (Z,X,e)$ if closed, then
  $\Sign X = 0$.
  We can also prove this directly by using Stokes' theorem, 
  and this method is enough to prove \fullref{maincor}.
  The proof given in \fullref{sec:yap} uses only \fullref{main:spin-cob} and Stokes' theorem.
\end{remark}


\subsection{Plan of the paper}
\label{sec:plan}
Here is the plan of the paper.

\rk{Preliminaries}
In \fullref{sec:notation}, 
we introduce notation and conventions. 
In \fullref{sec:construction},
we construct a $6$--dimensional closed spin $e$-manifold
$\a_{M} = (Y,Q,e_{M})$  such that
$Y \cong (M \x M)\#(-S^{3} \x S^{3})$ and $Q \cong M \amalg M \amalg (-M)$.

\rk{An involution}
Let $G = \left\{ 1,\i \right\}$ denote a multiplicative group of order $2$.
In \fullref{sec:involution}, we define a $G$--action on $(Y,Q)$  by using the permutation of
coordinates on $M \x M$ and $S^{3} \x S^{3}$.
We can regard $\i$ as an isomorphism between $-\a_{M}$ and $\a_{-M}$ (preserving the
orientation), namely, \fullref{main:spin-cob}~\eqref{item:reverse} holds.

\rk{Spin cobordism group of $e$-manifolds}
In \fullref{sec:spin-cobordism}, we prove \fullref{main:O},
more precisely, we give a short exact sequence 
\begin{equation*}
  0 \to \O_{4}^{spin}(BSpin(3)) \to \O_{6}^{spin}(K(\Q,2)) \to \O_{6}^{e,spin} \to 0,
\end{equation*}
which is isomorphic to
$0 \to 16\Z \oplus 4\Z \hookrightarrow \Q \oplus \Q \to (\Q/16\Z)\oplus(\Q/4\Z) \to 0$.
Here, $\O_{*}^{spin}$  denotes the spin cobordism group.
A pair $(W,e)$ of a closed spin $6$--manifold $W$ and $e \in H^{2}(W;\Q)$ represents an
element $[W,e] \in \O_{6}^{spin}(K(\Q,2))$, and the isomorphism (\fullref{lem:chi-xi})
\begin{equation*}
 \O_{6}^{spin}(K(\Q,2)) \to \Q \oplus \Q,\qquad
   [W,e] \mapsto
 \left( 
   \frac{1}{6} \int_{W}^{} p_{1}(TW)e - e^{3},\ 
 \frac{1}{2} \int_{W}^{} e^{3}
   \right)
\end{equation*}
induces the definition of $\Phi$.

\rk{Signature modulo $32$}
\label{sub:spin-cob}
In \fullref{sec:signature32},
we construct a certain closed spin $e$-manifold of the form
$
 \a_{M}' = (Y', \es, e_{M}')
$
such that $[\a_{M}'] = [\a_{M}]$ in $\O_{6}^{e,spin}$, and that $Y'$ has an orientation
reversing free $G$--action.
We show that, if $e_{M}'$ is the Poincar\'e dual of a $4$--submanifold $W$ of $Y'$, then the following
equivalence relation holds (\fullref{prop:W32}):
\begin{equation*}
  [\a_{M}]=0\ \text{(\fullref{main:spin-cob}~\eqref{item:spin-null})}
  \quad \Longleftrightarrow \quad \Sign W \equiv 0 \pmod{32}
\end{equation*}

\rk{$G$--vector bundle}
In \fullref{sec:congruence}, we prove
\fullref{main:spin-cob}~\eqref{item:spin-null}, by  constructing such a $W$.
This is done by assuming
the existence of an oriented vector bundle $F$, of rank $2$ over $Y'$ with a
$G$--action, such that
\begin{enumerate}[(i)]
  \item\label{item:a1} $e(F) = e_{M}'$ over $\Q$,
  \item\label{item:a2} $w_{i}(F/G) = w_{i}(TY'/G)$ in $H^{i}(Y'/G;\Z/2)$ for $i=1,2$,
\end{enumerate}
where $w_{i}$ denotes the $i$--th Stiefel--Whitney class.
Fix a $G$--equivariant smooth section $s \co Y' \to F$, and define
$W = \left\{ x \in Y' \ | \ s(x) = 0 \right\}$.
Then, the Poincar\'e dual of $W$ is $e_{M}'$ by \eqref{item:a1}.
The second property \eqref{item:a2} implies that the quotient $W/G$ is orientable and spinnable
smooth manifold.
By Rokhlin's theorem, we have $\Sign W = \pm 2 \Sign W/G \equiv 0 \pmod{32}$.
Hence, \fullref{main:spin-cob}~\eqref{item:spin-null} holds.
In \fullref{sec:vector-bundle}, we prove the existence of $F$ satisfying 
\eqref{item:a1} and \eqref{item:a2}.

In \fullref{sec:yap}, we give yet another direct proof of \fullref{maincor}.

\subsection{Remarks}
\label{remarks}

\begin{remark}
  \label{rem:direct}
The Casson invariant $\l(M)$ is roughly defined by measuring the oriented number of
irreducible representations of the fundamental group $\pi_{1}(M)$ in $SU(2)$, 
and so the geometric meaning is different from $\mu(M)$.
The relation \eqref{eq:l2} is proved by showing that the Dehn surgery formula for
$\l(M) \pmod{2}$ coincides with that of $\mu(M)$.
On the other hand, 
our proof does not require such formulas (or the fact that the Casson invariant is a
finite type invariant) in any step including the proof of \fullref{thm:sigma}.
Moreover, in this paper, we only need to consider the signature of $4$--manifolds or
the related characteristic classes to prove \fullref{maincor}.
Therefore, we might say that our proof is more direct.
\end{remark}
\begin{remark}
  The idea of the construction of $\a_{M}$ comes from the definition of the
  Kontsevich--Kuperberg--Thurston invariant $Z_{KKT}(M)$ for oriented rational homology
  $3$--spheres~\cite{kuperberg-thurston}~\cite{kontsevich-feynmann-diagram}, which is a
  universal real finite type invariant for integral homology spheres in the sense of 
  Ohtsuki~\cite{ohtsuki}, Habiro~\cite{habiro}, and 
  Goussarov~\cite{garoufalidis-goussarov-polyak}.
  A detailed review and an elementary proof for the invariance of $Z_{KKT}$
  is given by Lescop~\cite{lescop-kkt-construction}.
  The  degree one part $Z_{1}(M)$ of $Z_{KKT}(M)$ is equal to $\l_{W}(M)/4$
  (first proved by Kuperberg--Thurston~\cite{kuperberg-thurston} for integral homology $3$--spheres, 
  and later Lescop~\cite{lescop-split} extended this relation to all rational homology $3$--spheres).  
  By definition, $Z_{1}(M)$ is described as an integral over the configuration space
  $\operatorname{Conf}_{2}(M') = M' \x M' \setminus M_{\D}'$ of two points on $M' = M \setminus \left\{ x_{0}
  \right\}$, where $M_{\D}' \subset M' \x M'$ is the diagonal submanifold.
\end{remark}
\begin{remark}
  By the construction of $(Y,Q)$ (\fullref{sec:construction}),
  the complement $Y \setminus Q$ is nothing but the union of the two configuration spaces 
  $\operatorname{Conf}_{2}(M')$ and $-\operatorname{Conf}_{2}(\R^{3})$, and 
  the $G$--action on $Y \setminus Q$ corresponds to the permutation of coordinates on the
  configuration spaces.
  To be brief, the invariant $\s(\a_{M})$ measures the difference between the manifolds 
  $\operatorname{Conf}_{2}(M')$ and $\operatorname{Conf}_{2}(\R^{3})$ (equipped with some
  second cohomology classes) by using the signature of $4$--manifolds.
\end{remark}
\begin{remark}
  If $M$ is an oriented rational homology $3$--sphere, then we can define a
  $6$--dimensional closed $e$-manifold $\a_{M} = (Y,Q,e_{M})$ in exactly the same way as for
  integral homology $3$--spheres.   
  The isomorphism class of $\a_{M}$ is a topological invariant of $M$ (this can be proved in
  the same way as the proof of \fullref{prop:top-inv}), and therefore, the rational
  number $\s(\a_{M}) \in \Q$ is a topological invariant of $M$.   
  In a future paper\footnote{T. Moriyama, \textit{Casson--Walker invariant the signature of
  spin $4$--manifolds}, in preparation}, we will prove that $\s(\a_{M})$ is equal to the
  Casson--Walker invariant $\l_{W}(M)$ up to multiplication by a constant. 
\end{remark}

\textit{Acknowledgments.}
  I would like to thank Professor Mikio Furuta, Toshitake Kohno, and Christine Lescop for their
  advice and support.

\section{Notation}
\label{sec:notation}

We follow the notation introduced in~\cite{moriyama:emb63}.
All manifolds are assumed to be compact, smooth, and oriented unless otherwise stated, and
we use the ``outward normal first'' convention for boundary orientation of manifolds.

For an oriented real vector bundle $E$ of rank $3$ over a manifold
$X$, we denote the associated unit sphere bundle by
$\rho_{E} \co S(E) \to X$,
and let $F_{E} \subset TS(E)$ denote the vertical tangent subbundle of $S(E)$ with respect to $\rho_{E}$.
The orientations of $F_{E}$ and $S(E)$ are given by the isomorphisms
$\rho_{E}^*E \cong \R_{E} \oplus F_{E}$ and $TS(E) \cong \rho_{E}^*TX \oplus
F_{E}$,
where $\R_{E} \subset \rho_{E}^*E$ is the tautological real line bundle of $E$ over $S(E)$.
Consequently, the Euler class
\[
e(F_{E}) \in H^{2}(S(E);\Z)
\]
of $F_{E}$ is defined.

Next, let $(Z,X)$, $Z \supset X$, be a pair of manifolds, and
we assume that $X$ is properly embedded in $Z$ and the codimension is $3$.
Throughout this paper, we always impose these assumptions for all pairs of manifolds.
Denote by $\nu_{X}$ the normal bundle of $X$, which can be identified with a tubular
neighborhood of $X$ so that $X \subset \nu_{X} \subset Z$.
For simplicity, we write
\begin{equation*}
  \HX = S(\nu_{X}), \qquad
  \rho_{X} = \rho_{\nu_{X}} \co \HX \to X, \qquad
  F_{X} = F_{\nu_{X}},\qquad
  Z_{X} = Z \setminus U_{X},
\end{equation*}
where $U_{X}$ is the total space of the open unit disk bundle of $\nu_{X}$.

If we denote by $(W,V) = \bd{(Z,X)}$ the boundary pair of $(Z,X)$, then
we can define $\nu_{V}$, $F_{V}$, $\HV$, $\rho_{V}$, $W_{V}$,
etc.~in exactly the same way as above, and we have
$\bd{\HX} = \HV$ and $e(F_{X})|_{\HV} = e(F_{V})$.

In line with our orientation conventions,
if $\dim Z = 7$ (and so $\dim W = 6$), then
the oriented boundaries of $Z_{X}$ and $W_{V}$ are given as follows:
\begin{equation*}
  \bd{Z_{X}} = W_{V} \cup (- \HX),\qquad
  \bd{W_{V}} = \HV
\end{equation*}
Note that $Z_{X}$ have the corner $\HV$ which is empty when $X$ is closed.
By definition, $e \in H^{2}(Z \setminus X; \Q)$ is an $e$-class $(Z,X)$ if, and only if, 
$e|_{\HX} = e(F_{X})$ over $\Q$.
See \cite{moriyama:emb63} for more detailed description.

\section{\texorpdfstring{Construction of $\a_{M}$}{Construction of aM}}
\label{sec:construction}

Let $M$ be an oriented integral homology $3$--sphere.
In this section, we give a precise construction of the $e$-manifold $\a_{M} = (Y,Q,e_{M})$.

Identify the $3$-sphere $S^{3}$ with the
one-point compactification $\R^{3} \amalg \left\{ \infty \right\}$
of the Euclidean $3$--space $\R^{3}$ by adding one point $\infty$ at
infinity.
We can regard $\R^{3} \x \R^{3}$ as an open submanifold of $S^{3} \x S^{3}$ such that
$S^{3} \x S^{3} = (\R^{3} \x \R^{3}) \amalg (S^{3}_{1} \cup S^{3}_{2})$, where
\begin{equation}
  \label{eq:bouquet}
  S^{3}_{1} = S^{3} \x \left\{ \infty \right\},\quad
  S^{3}_{2} = \left\{ \infty \right\} \x S^{3}.
\end{equation}


Fix a base point $x_{0} \in M$ and 
a smooth oriented local coordinates $\vp \co U \to \R^{3}$ such that $\vp(x_{0}) = 0$.
We shall assume that $U$ is sufficiently small, so that,
for any such a local coordinates $\vp' \co U' \to \R^{3}$, there exists an orientation
preserving  smooth diffeomorphism $h \co M \to M$ such that $h(U) = U'$ and $\vp'
h|_{U} = \vp \co U \to \R^{3}$. 
Set $P = \left\{ (x_{0}, x_{0}) \right\}$.
We define
\begin{align*}
  Y &= 
  \left( M \x M \setminus P \right)
  \cup_{g_{\vp}}
  \left( S^{3} \x S^{3} \setminus \left\{  (0,0)\right\} \right)
\end{align*}
to be the oriented closed $6$-manifold obtained by gluing $U \x U \setminus P$ and $\R^{3} \x
\R^{3} \setminus \left\{ (0,0) \right\}$   
by using the gluing map 
$g_{\vp} \co U \x U \setminus P \to \R^{3} \x \R^{3} \setminus \left\{ (0,0) \right\}$ defined
by
\begin{equation}
  g_{\vp}(x,y) =
  \frac{(\vp(x), \vp(y))}{\norm{(\vp(y), \vp(y))}^{2}},
  \quad
  (x,y) \in U \x U \setminus P,
  \label{eq:gluing-map}
\end{equation}
where $\norm{\ }$ is the standard norm of $\R^{3} \x \R^{3} = \R^{6}$.
By definition, $Y \cong (M \x M) \# ( -S^{3} \x S^{3})$.
\begin{remark}
  We have to remember that we use $g_{\vp}$ to perform the gluing, so that we can define an
  involution on $Y$ in \fullref{sec:involution}.
\end{remark}
We can regard $M \x M \setminus P$ and $-S^{3} \x S^{3}
\setminus \left\{ (0,0) \right\}$ as open submanifolds of $Y$.
The closure of $M \x M \setminus P$ in $Y$ is $Y$ itself,
and so this procedure to obtain $Y$ from $M \x M$ is a kind of blow--up that 
replaces one point $P$ to the bouquet $S^{3}_{1} \cup S^{3}_{2}$, where note that $S^{3}_{1} \cap
S^{3}_{2} = \left\{ (\infty,\infty) \right\}$.

We have the following three $3$--submanifolds $M_{i}'$ ($i=1,2,3$) of $Y$:
\begin{equation*}
  M_{1}' = (M \x \left\{ x_{0} \right\}) \setminus P,\quad
  M_{2}' = (\left\{ x_{0} \right\} \x M) \setminus P,\quad
  M_{3}' = M_{\D} \setminus P
\end{equation*}
Here, $M_{\D} \subset M \x M$ is the diagonal submanifold.
The closure of $M_{i}'$ in $Y$ will be denote by $M_{i}$,
which is smoothly embedded $3$--submanifold of $Y$ such that
\begin{gather*}
  M_{1} = M_{1}' \amalg \left\{ (\infty,0) \right\},\quad
  M_{2} = M_{2}' \amalg \left\{ (0,\infty) \right\},\quad
  M_{3} = M_{3}' \amalg \left\{ (\infty,\infty) \right\},\\
  M_{i} \cong M,\quad
  M_{i} \cap M_{j} = \es\quad \text{($i \neq j$),}.
\end{gather*}
We then define 
\[
	Q = M_{1} \cup M_{2} \cup (-M_{3}),
\]
which is a $3$--submanifold of $Y$, see \fullref{fig:conn-sum}.
\begin{figure}[tb]
\begin{center}
  \begin{picture}(0,0)%
  \includegraphics{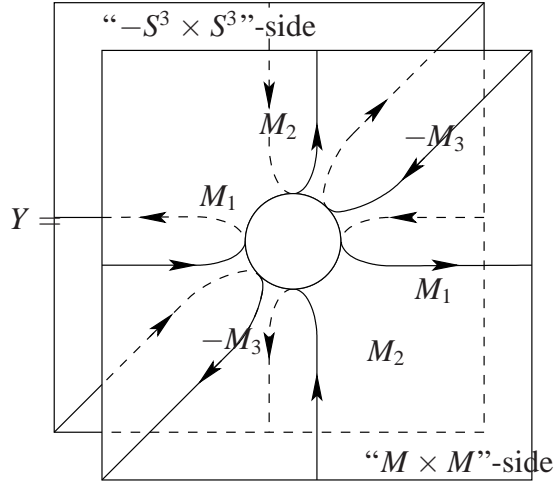} %
  \end{picture} %
  \setlength{\unitlength}{3947sp} %
  \begingroup\makeatletter\ifx\SetFigFont\undefined%
  \gdef\SetFigFont#1#2#3#4#5{%
	\reset@font\fontsize{#1}{#2pt} %
	\fontfamily{#3}\fontseries{#4}\fontshape{#5} %
	\selectfont} %
  \fi\endgroup%
  \begin{picture}(3537,3024)(5176,-4123)
  \put(5176,-2536){\makebox(0,0)[lb]{\smash{{\SetFigFont{12}{14.4}{\rmdefault}{\mddefault}{\updefault}{$Y = $} %
  }}}}
  \put(7651,-2011){\makebox(0,0)[lb]{\smash{{\SetFigFont{12}{14.4}{\rmdefault}{\mddefault}{\updefault}{$-M_3$} %
  }}}}
  \put(6376,-3286){\makebox(0,0)[lb]{\smash{{\SetFigFont{12}{14.4}{\rmdefault}{\mddefault}{\updefault}{$-M_3$} %
  }}}}
  \put(7426,-3361){\makebox(0,0)[lb]{\smash{{\SetFigFont{12}{14.4}{\rmdefault}{\mddefault}{\updefault}{$M_2$} %
  }}}}
  \put(6751,-1936){\makebox(0,0)[lb]{\smash{{\SetFigFont{12}{14.4}{\rmdefault}{\mddefault}{\updefault}{$M_2$} %
  }}}}
  \put(6376,-2386){\makebox(0,0)[lb]{\smash{{\SetFigFont{12}{14.4}{\rmdefault}{\mddefault}{\updefault}{$M_1$} %
  }}}}
  \put(7726,-2971){\makebox(0,0)[lb]{\smash{{\SetFigFont{12}{14.4}{\rmdefault}{\mddefault}{\updefault}{$M_1$} %
  }}}}
  \put(7424,-4061){\makebox(0,0)[lb]{\smash{{\SetFigFont{12}{14.4}{\rmdefault}{\mddefault}{\updefault}{``$M \times M$''-side} %
  }}}}
  \put(5776,-1322){\makebox(0,0)[lb]{\smash{{\SetFigFont{12}{14.4}{\rmdefault}{\mddefault}{\updefault}{``$-S^3\times S^3$''-side} %
  }}}}
  \end{picture} %
  \hspace{1.5cm}
  \hspace{1cm}
  \caption{The manifold pair $(Y,Q)$.}
  \label{fig:conn-sum}
\end{center}
\end{figure}
\begin{notation}
We will sometimes write $(Y(M),Q(M))$, instead of just $(Y,Q)$, to emphasize
that this is constructed from $M$.  
\end{notation}
Two (smooth oriented) manifold pairs $(W,V)$ and $(W',V')$ are said to be \emph{isomorphic} if
there exists an orientation preserving diffeomorphism $f \co W \to W'$ such that $f(V) = V'$ as
an oriented submanifold.
\begin{lemma}
  \label{lem:iso}
  The isomorphism class of the pair $(Y,Q)$ of manifolds depends only
   on the topological type and the orientation of $M$. In particular, it does 
   not depend on $x_{0}$ or $\vp$.    
\end{lemma}
\begin{proof}
  Let $V_{i}$ be an oriented integral homology $3$--sphere with a base point $x_{i}$ 
  and with  an orientation preserving local coordinates $\vp_{i} \co U_{i} \to \R^{3}$ such that
  $\vp_{i}(x_{i}) = 0$ ($i=1,2$).  
  Then, we can define the pair of manifolds 
  \begin{equation*}
	(Y_{i}, Q_{i}) = (Y(V_{i}), Q(V_{i})),
  \end{equation*}
  by using the gluing map $g_{\vp_{i}}$ as in \eqref{eq:gluing-map}.

  Assume $V_{1} \cong V_{2}$ as an oriented topological manifold.
  Since the topological and the smooth categories are equivalent in dimension three, 
  there exists an orientation preserving diffeomorphism $h \co V_{1} \to V_{2}$ such that
  $h(U_{1}) = U_{2}$ and $\vp_{1} = \vp_{2} h|_{U_{1}}$. 
  Therefore, $g_{\vp_{1}}$ coincides with
  \begin{equation*}
	g_{\vp_{2}}(h \x h)|_{U_{1} \x U_{1} \setminus P_{1}} 
	\co 
	U_{1} \x U_{1} \setminus P_{1} \to
	\R^{3} \x \R^{3} \setminus \left\{ (0,0) \right\},
  \end{equation*}
  where $P_{i} = \left\{ (x_{i}, x_{i}) \right\}$.
  Hence, the diffeomorphism 
  \begin{equation*}
	h \x h \co V_{1} \x V_{1} \setminus P_{1} \to V_{2} \x V_{2} \setminus P_{2}
  \end{equation*}
  uniquely extends to an orientation preserving diffeomorphism $Y_{1} \to Y_{2}$
  which sends $Q_{1}$ onto $Q_{2}$.
  Hence,
  $(Y_{1}, Q_{1})$ and  $(Y_{2}, Q_{2})$ are isomorphic.
\end{proof}

\begin{lemma}\label{lem:eM}
  The pair $(Y,Q)$ admits a unique $e$-class.
\end{lemma}
\begin{proof}
  In general,
  a closed manifold pair $(W,V)$ of dimensions $6$ and $3$ admits
  a unique $e$-class if it satisfies the following two conditions~\cite[Proposition~6.1~(5)]{moriyama:emb63}:
  \begin{enumerate}
	\item \label{item:iso} The restriction $H^{2}(W;\Q) \to H^{2}(V;\Q)$ is isomorphic.
	\item \label{item:null-h} $[V] = 0$ in $H_{3}(W;\Q)$, where $[V]$ is the fundamental
	  homology class of $V$.
  \end{enumerate}

  Since the first and the second betti--numbers of $Y$ and $Q$ vanish, $(Y,Q)$ satisfies~\eqref{item:iso}.
  By the same reason, we have
  $
	[M_{1}] + [M_{2}] = [M_{3}]
  $
  in $H_{3}(Y;\Q)$.  Consequently, $[Q] = [M_{1}]+[M_{2}] - [M_{3}] = 0$, namely, $(Y,Q)$
  satisfies~\eqref{item:null-h}.
  Hence, $(Y,Q)$ admits an unique $e$-class.
\end{proof}
We denote by
$
  e_{M} \in H^{2}(Y \setminus Q; \Q)
$
the unique $e$-class of $(Y,Q)$, and we define
\[
  \a_{M} = (Y,Q,e_{M})
\]
which is a $6$--dimensional closed spin $e$-manifold.
\begin{proposition}\label{prop:top-inv}
  The isomorphism class of $\a_{M}$ depends only on the topological type and the orientation of $M$.
\end{proposition}
\begin{proof}
  In general, if there is an isomorphism $f \co (W,V) \to (W',V')$ of pair of manifolds of codimension
  $3$, then the pull--back $f^{*} \co H^2(W' \setminus V';\Q) \to H^{2}(W \setminus
  V;\Q)$ maps an $e$-class to an $e$-class.
  Thus, by \fullref{lem:iso} and \fullref{lem:eM}, 
  the isomorphism class of $\a_{M}$ depends only on the topological type and the orientation of
  $M$.
\end{proof}

\section{\texorpdfstring{An involution}{An involution}}
\label{sec:involution}
Let $G = \left\{ 1,\i \right\}$ be a multiplicative group of order two.
Let $M$, $g_{\vp}$,  and $\a_{M} = (Y,Q,e_{M})$ be as in \fullref{sec:construction}.
In this section, 
we prove \fullref{main:spin-cob}~\eqref{item:reverse}, by 
constructing a $G$--action on $\a_{M}$ which reverses the orientation of $Y$.
\begin{remark}
  In this paper, $G$--actions we use may reverses the orientation of manifolds.
  Therefore,  in this paper, 
  a \emph{$G$--manifold} (resp.~\emph{$G$--vector bundle}) will mean an oriented manifold
  (resp.~vector bundle) with a smooth $G$--action which may reverses the orientation unless
  otherwise stated.
\end{remark}

The group $G$ acts on $M \x M$ and $S^{3} \x S^{3}$ by permuting coordinates.
Since the gluing map $g_{\vp}$ commutes with the $G$-action, $Y$
has the induced smooth $G$-action.
It is easy to check that $\i(M_{1}) = M_{2}$,
and that the fixed point set of the action on $Y$ is $M_{3}$.
Consequently, $\i(Q) = Q$ as an oriented submanifold.
Note that the involution $\i$ reverses the orientation of $Y$ and preserves that of $Q$.	
Thus, we can regard $\i$ as an isomorphism
\begin{equation}
  \label{eq:iota}
  \i \co (-Y,-Q) \to (Y,-Q)
\end{equation}
of pair of (oriented) manifolds.  
\begin{lemma}
  \label{lem:ie} 
  \textup{\fullref{main:spin-cob}~\eqref{item:reverse}} holds, namely, $\a_{-M} \cong -\a_{M}$.
\end{lemma}
\begin{proof}
  We shall identify $(Y(-M), Q(-M))$ with $(Y, -Q)$ which admits a unique
  $e$-class $-e_{M}$ by \fullref{lem:eM}, and hence, 
  \[
  	\a_{-M} = (Y,-Q,-e_{M}).
  \]
  The homomorphism $\i^* \co H^{2}(Y \setminus Q;\Q) \to H^{2}(Y \setminus Q; \Q)$ induced
  from \eqref{eq:iota} maps an $e$-class of $(Y,-Q)$ to an $e$-class of $(-Y,-Q)$, 
  which means $\i^*(-e_{M}) = e_{M}$.
  Thus, $\i$ is an isomorphism from $-\a_M$ to $\a_{-M}$. 
\end{proof}

\section{\texorpdfstring{Spin cobordism group of $e$-manifolds}{Spin cobordism group of
e-manifolds}}
\label{sec:spin-cobordism}
In~\cite{moriyama:emb63}, we proved that 
there is an isomorphism $\O_{6}^{e} \cong (\Q/\Z)^{\oplus 2}$, where $\O_{6}^{e}$ is the
cobordism group of $6$--dimensional $e$-manifolds.
In this section, we prove 
that there is a similar isomorphism $\Oes \cong (\Q/16\Z) \oplus
(\Q/4\Z)$.
The only difference between the two proofs is that spin structures are not considered in
\cite{moriyama:emb63},
and the essential ideas behind the proofs are the same.

\subsection{\texorpdfstring{Preliminaries: $K(\Q,2)$ and $BSpin(3)$}{Preliminaries: K(Q,2) and
BSpin(3)}}
\label{sec:preliminaries}
Let $K(\Q,2)$ be the Eilenberg--MacLane space of type $(\Q,2)$.
The reduced homology group of $K(\Q,2)$ is given as follows (cf.~\cite{GM}):
\begin{equation}
  \Tilde{H}_{k}(K(\Q,2);\Z) \cong
  \begin{cases}
	\Q & \text{if $k> 0$ and $k \equiv 0 \pmod{2}$}\\
	0 & \text{otherwise}
  \end{cases}
  \label{eq:KQ2}
\end{equation}
The cohomology group $H^{2k}(K(\Q,2); \Q) \cong \Q$ ($k \geq 0$) is generated by the $k$--th power $a_{1}^{k}$ of the
dual element $a_{1} \in H^{2}(K(\Q,2);\Q)$ of $1 \in \pi_{2}(K(\Q,2))$.

Let $BSpin(3)$ be the classifying space of the Lie group $Spin(3)$.  
Since $BSpin(3)$ is homotopy equivalent  to the infinite dimensional quaternionic projective space
$\HP^{\infty}$,  the following isomorphism holds:
\begin{equation}
  H_{k}(BSpin(3);\Z) \cong
  \begin{cases}
  \Z & \text{if $k \geq 0$ and $k \equiv 0 \pmod{4}$}\\
  0 & \text{otherwise}
  \end{cases}
  \label{eq:BSpin3}
\end{equation}

We can assume that $K(\Q,2)$ and $BSpin(3)$ have structures of CW--complexes.
Let $\Os_{*}(V)$ denote the spin cobordism group of a CW--complex $V$.
In low--dimensions, the spin cobordism group $\Os_{*} = \Os_{*}(pt)$ of one point  $pt$ is given
as follows (cf.~\cite{spin-geom}):
\begin{equation}
  \begin{array}{c|ccccccc}
	k & 0 & 1 & 2 & 3 & 4 & 5 & 6\\
	\hline
	\Os_{k} & \Z & \Z/2 & \Z/2 & 0 & \Z & 0 & 0
  \end{array}
  \label{tbl:Ospin}
\end{equation}
In general, the Atiyah--Hirzebruch spectral sequence $E^{n}_{p,q}(Y)$ for $\Os_{*}(Y)$ converges
 (cf.~\cite{switzer}):
\begin{equation*}
  E^{2}_{p,q} = H_{p}(Y;\Os_{q}) \Longrightarrow \Os_{p+q}(Y)
\end{equation*}
The following lemma is an easy application of the Atiyah--Hirzebruch Spectral sequence.
\begin{lemma}
  \label{lem:O-iso}
  The following isomorphisms hold:
  \begin{equation*}
	\Os_{6}(K(\Q,2)) \cong \Q^{\oplus 2},\quad
	\Os_{4}(BSpin(3)) \cong \Z^{\oplus 2}
  \end{equation*}
\end{lemma}
\begin{proof}
  We use \eqref{eq:KQ2}, \eqref{eq:BSpin3}, and \eqref{tbl:Ospin}
  to prove this lemma.
  The Atiyah--Hirzebruch spectral sequence $E^{n}_{p,q} = E^{n}_{p,q}(K(\Q,2))$ for
  $\O_{*}^{spin}(K(\Q,2))$ converges on
  the $E^{2}$-stage within the range $p+q \leq 6$, and so $E^{\infty}_{p,q} \cong E^{2}_{p,q}$
  in the same range.
  Consequently, we have
  \begin{equation*}
	E^{\infty}_{p,6-p} \cong
	\begin{cases}
	  \Q & \text{if $p=6,2$},\\
	  0 & \text{otherwise},
	\end{cases}
  \end{equation*}
  and therefore, $\Os_{6}(K(\Q,2)) \cong \Q^{\oplus 2}$.

  Similarly, the spectral sequence $F^{n}_{p,q} = E^{n}_{p,q}(BSpin(3))$
  converges on the $F^{2}$-stage in the range $p+q \leq 4$, and 
  \begin{equation*}
	F^{\infty}_{p,4-p} \cong
	\begin{cases}
	  \Z & \text{if $p=4,0$,}\\
	  0 & \text{otherwise.}
	\end{cases}
  \end{equation*}
  Thus, $\Os_{4}(BSpin(3)) \cong \Z^{\oplus 2}$.
\end{proof}

\subsection{\texorpdfstring{Spin cobordism groups of $BSpin(3)$ and
$K(\Q,2)$}{Spin cobordism groups of BSpin(3) and K(Q,2)}}
\label{sec:BSpin3-KQ2}
We define three homomorphisms $\chi$, $\xi$, and $\up$ as follows.
A pair $(W,e)$ of a closed spin $6$--manifold $W$ and a cohomology class $e \in H^{2}(W;\Q)$
represents an element $[W,e] \in \Os_{6}(K(\Q,2))$.  
Here, we identify $e$ with the homotopy class of a map $f \co W \to K(\Q,2)$ such that
$f^*a_{1} = e$.
Define a homomorphism
\[
	\chi \co \Os_{6}(K(\Q,2)) \to \Q^{\oplus 2}
\]
by $\chi([W,e]) = \left(\chi_{1}(W,e), \chi_{2}(W,e) \right)$, where
\begin{align*}
  \chi_{1}(W,e) &= \frac{1}{6}\int_{W}p_{1}(TW)\,e - e^{3} \in \Q,\\
  \chi_{2}(W,e) &= \frac{1}{2} \int_{W}e^{3} \in \Q.
\end{align*}
Similarly, a pair $(X,E)$ of a closed spin $4$--manifold $X$ and a spin vector bundle $E$ of rank $3$
over $X$ represents an element $[X,E] \in \Os_{4}(BSpin(3))$.  Here, we identify the
isomorphism class of $E$ with
the homotopy class of the classifying map $X \to BSpin(3)$ of $E$.
Note that $p_{1}(E) \equiv 0 \pmod{4}$ (since $E$ is spin),
and that
\begin{equation}
  \Sign X \equiv 0 \pmod{16}
  \label{eq:rokhlin}
\end{equation}
by the Rokhlin's theorem.
Define a homomorphism 
\begin{equation*}
	\xi \co \Os_{4}(BSpin(3)) \to 16\Z \oplus 4\Z
\end{equation*}
\begin{equation*}
  \tag*{\text{by}}
  \xi([X,E]) = \left(  \Sign X, \int_{X}p_{1}(E)\right).
\end{equation*}
We will see soon that $\chi$ and $\xi$ are isomorphic (\fullref{lem:chi-xi}).
We define a homomorphism
\[
\up \co \Os_{4}(BSpin(3)) \to \Os_{6}(K(\Q,2))
\]
by
$
\up([X,E]) = [S(E), e(F_{E})].
$

Now, for a pair $(X,E)$ representing an element in $\Os_{4}(BSpin(3))$,
the characteristic classes of the vector bundles $E$, $F_{E}$, $TX$, and $TS(E)$
satisfy the following relations:
\begin{align}
  \label{eq:pe1}
  e(F_{E})^{2} &= p_{1}(F_{E}) = \rho_{E}^*p_{1}(E)\\
  \label{eq:pe2}
  &\equiv p_{1}(TS(E)) - \rho_{E}^*p_{1}(TX)\quad \text{(modulo $2$--torsion
  elements)},\\
  \label{eq:gysin}
  {\rho_{E}}_{!} e(F_{E}) &=2
\end{align}
Here, ${\rho_{E}}_{!} \co H^{2}(S(E);\Z) \to H^{0}(X;\Z)$ is the
Gysin homomorphism of $\rho_{E}$, and  $2 \in
H^{0}(X;\Z)$ denotes the element given by the constant function on $X$ with
the value $2$ ($=$ Euler characteristic of $S^{2}$).
The Hirzebruch signature theorem states that
\begin{equation}
  \label{eq:hirzebruch}
  \Sign X = \frac{1}{3}\int_{X}p_{1}(TX).
\end{equation}
The next two lemmas are easy to prove.
\begin{lemma}
  \label{lem:cu}
  $\chi \up = \xi$.  In other words, for any pair $(X,E)$ of closed spin $4$--manifold $X$ and a spin
  vector bundle $E$ of rank $3$ over $X$, we have
  \begin{equation*}
	\label{eq:SO}
	\chi([S(E), e(F_{E})]) = \left( \Sign X,\  \int_{X}p_{1}(E) \right).
  \end{equation*}
\end{lemma}
\begin{proof}
  This follows from the formulas \eqref{eq:pe1}, \eqref{eq:pe2},
  \eqref{eq:gysin}, and \eqref{eq:hirzebruch}.
  In fact, we have
  $p_{1}(TS(E))e(F_{E}) - e(F_{E})^{2} = \rho_{E}^*p_{1}(TX) e(F_{E})$, and so
 \begin{equation*}
   \chi_{1}(S(E), e(F_{E})) = \frac{1}{6}\int_{S(E)}^{} \rho_{E}^*p_{1}(TX) e(F_{E})
   = \frac{1}{3} \int_{X}^{}p_{1}(TX) = \Sign X.
 \end{equation*}
 Similarly, we have
 \begin{equation*}
   \chi_{2}(S(E),e(F_{E})) = \frac{1}{2} \int_{S(E)}^{} \rho_{X}^*p_{1}(E) e(F_{E}) = 
   \int_{X}^{}p_{1}(E).
   \proved
 \end{equation*}
\end{proof}

\begin{lemma}
  \label{lem:chi-xi}
  The homomorphisms $\chi$ and $\xi$ are isomorphic.
\end{lemma}
\begin{proof}
  The $K3$-manifold $K3$ is a closed spin $4$--manifold with the signature $-16$.
  There exists an oriented spin vector bundle $E$ of rank $3$ over $S^{4}$ such that
  $p_{1}(E) = 4$ in $H^{4}(S^{4};\Z) \cong \Z$.
  We define two elements $u_{1}, u_{2} \in \O_{4}^{spin}(BSpin(3))$ as follows:
  \begin{equation*}
	u_{1} = [K3, K3 \x \R^{3}],\quad
    u_{2} = [S^{4}, E]
  \end{equation*}
  Then, $\xi(u_{1}) = (-16,0)$ and $\xi(u_{2}) = (0,4)$ by definition.
  Therefore, $\Image \chi = (16\Z) \oplus (4\Z)$.
  In particular, $\xi$ is a surjective homomorphism from $\Os_{4}(BSpin(3)) \cong \Z^{\oplus 2}$
  (\fullref{lem:O-iso}) to $16\Z \oplus 4\Z$.
  This means that $\xi$ is an isomorphism.

  Similarly, we have
  $\chi(\up(u_{1})) = (-16,0)$ and $\chi(\up(u_{1})) = (0, 4)$ by \fullref{lem:cu},
  and these two elements form a basis of the vector space $\Q^{\oplus 2}$ over $\Q$.
  Therefore, $\chi$ is a linear homomorphism from $\Os_{6}(K(\Q),2) \cong \Q^{\oplus
  2}$ (\fullref{lem:O-iso}) to $\Q^{\oplus 2}$ of rank $2$.
  This means that $\chi$ is an isomorphism.
\end{proof}
\begin{proposition}
  \label{prop:exact-O}
  The sequence of the homomorphisms 
  \begin{equation*}
	\label{eq:exact-O}
	0 \to \Os_{4}(BSpin(3)) \xrightarrow{\up} \Os_{6}(K(\Q,2)) \xrightarrow{\chi'}
	(\Q/16\Z)\oplus (\Q/4\Z) \to 0
  \end{equation*}
  is exact, where $\chi' = \chi \mod{16\Z \oplus 4\Z}$.
\end{proposition}
\begin{proof}
  This follows from that, the diagram
  \begin{equation*}
	\begin{CD}
	  \Os_{4}(BSpin(3)) @>{\up}>> \Os_{4}(K(\Q,2))\\
	  @V{\xi}V{\cong}V @V{\chi}V{\cong}V\\
	  (16\Z) \oplus (4\Z) @>\text{inclusion}>> \Q^{\oplus 2}
	\end{CD}
  \end{equation*}
  commutes (\fullref{lem:cu}) and the vertical arrows are isomorphic (\fullref{lem:chi-xi}).
\end{proof}

\subsection{\texorpdfstring{An extension}{An extension}}
\label{sec:cob-e}
Let us consider the homomorphism
\[
\pi \co \Os_{6}(K(\Q,2)) \to \Oes
\]
defined by $\pi([W,e]) = [W,\es,e]$ for $[W,e] \in \Os_{6}(K(\Q,2))$.
We can prove that $\pi$ is surjective as follows.

Let $\a = (W,V,e)$ be a $6$--dimensional closed spin $e$-manifold.
The normal bundle $\nu_{V}$ of $V$ is trivial, because it is spin.
We fix a trivialization of $\nu_{V}$, so that a closed tubular neighborhood of $V$ is identified
with $V \x D^{3}$ such that $V \x S^{2} = \HV$.
Let $X$ be a spin $4$--manifold such that $\bd{X} = V$.

Let $p \co X \x S^{2} \to S^{2}$ be the projection, and $e(TS^{2})$ the Euler class of
$S^{2}$.
Two spin manifolds $W_{V}$ and $X \x S^{2}$ have the common spin boundary
$\bd{W_{V}} = \HV = \bd{(X \x S^{2})}$, and the cohomology classes $e$ and $p^*e(TS^{2})$
restrict to the same element $e(F_{V})$ on $\HV$ over $\Q$.
Let us consider the closed oriented spin $6$--manifold
\begin{equation}
  W' = W_{V} \cup_{\HV} (-X \x S^{2})
  \label{eq:Wp}
\end{equation}
obtained from $W_{V}$ and $-X \x S^{2}$ by gluing along the common boundaries.
There exists a cohomology class $e' \in H^{2}(W';\Q)$ such that
\begin{equation}
  e'|_{W_{V}} = e|_{W_{V}},\quad
  e'|_{X \x S^{2}} = p^*e(TS^{2}).
  \label{eq:ep}
\end{equation}
We obtain a 
$6$--dimensional closed spin $e$-manifold
$\a' = (W',\es,e')$
and a cobordism class
$[W',e'] \in \Os_{6}(K(\Q,2))$ such that $\pi([W',e']) = [\a']$.
\begin{proposition}
  \label{prop:pi-onto}
  Let $\a$, $X$, and $\a' = (W',\es,e')$ be as above.
  Then, there exists a $7$--dimensional spin $e$-manifold of the form $\b = (Z,X,\Te)$ 
  for some spin $7$--manifold $Z$ and $\Te \in H^{2}(Z \setminus X;\Q)$ such that
  $\bd{\b} \cong \a \amalg (-\a')$.
  In particular, $\pi([W',e']) = [\a]$ in $\Oes$.
  Consequently, the homomorphism $\pi$ is surjective.
\end{proposition}
\begin{proof}
  Let $I = [0,1]$ be the interval.
  In this proof, for a subset $A \subset W$, we write $A_{t} = \left\{ t \right\} \x A \subset I \x W$ for $t = 0, 1$.

  Gluing the $7$--manifolds $I \x W$ and $X \x D^{3}$ along $D(\nu_{V})_{0} \subset W_{0}$ and $V \x
  D^{3}\subset \bd{(X \x D^{3})}$ by using the identity map,
  we obtain a spin $7$--manifold
  \begin{equation*}
	Z = (X \x D^{3}) \cup_{(V \x D^{3})_{0}} \left( I \x W \right)
  \end{equation*}
  with the boundary
  \begin{equation*}
	\begin{split}
	  \bd{Z} &= W_{1} \amalg \left( (X \x S^{2}) \cup_{\HV_{0}} (- (W_{V})_{0})  \right)\\
	  &\cong W \amalg (-W'),
	\end{split}
  \end{equation*}
  and we shall assume that $\bd{Z}$ is smooth after the corner $\HV_{0}$ is rounded.
  The spin $4$--submanifold
  \begin{equation*}
	(X \x \left\{ 0 \right\}) \cup_{V_{0}} \left( I \x V \right) \subset Z
  \end{equation*}
  is properly embedded in $Z$, and is bounded by $V_{1}$.
  We will rewrite $X \cup_{V_{0}} \left( I \x V \right)$ as $X$ and identify $\bd{Z}$ with $W \amalg (-W')$,
  so that
  \begin{equation*}
	\bd{(Z,X)} = (W,V) \amalg (-W',\es)
  \end{equation*}
  as a spin manifold pair.

  Now, all that is left to do is to show the existence of an $e$-class of
  $(Z,X)$ restricting to $e$ and $e'$ on the boundary components.
  Since the inclusion $W' \hookrightarrow Z \setminus X$ is homotopy
  equivalence, there exists a cohomology class $\Te \in H^{2}(Z \setminus X;\Q)$ of
  $(Z,X)$ such that $\Te|_{W'} = e'$.
  By construction, $\Te$ is an $e$-class of $(Z,X)$ such that $\Te|_{W \setminus V} = e$.
  Hence, we obtain a $7$--dimensional spin $e$-manifold $ \b = (Z,X,\Te)$ bounded by
  \begin{equation*}
	\bd{\b} = (W,V,\Te|_{W \setminus V}) \amalg (-W',\es,\Te|_{W'})
	= \a \amalg (-\a').
	\proved
  \end{equation*}
\end{proof}

\subsection{\texorpdfstring{Proof of \fullref{main:O}}{Proof of Theorem~\ref{main:O}}}
\label{sec:isomorphism}
In this subsection, we prove \fullref{main:O}.
By \fullref{prop:pi-onto}, 
we can use the formula \eqref{eq:Phi2} to define the homomorphism $\Phi \co \O_{6}^{e,spin} \to
(\Q/16\Z) \oplus (\Q/4\Z)$.
The first thing we have to do is to show that $\Phi$ is well--defined.
\begin{lemma}
  The homomorphism $\Phi \co \Oes \to (\Q/16\Z) \oplus (\Q/4\Z)$ is well--defined.
  \label{lem:Phi}
\end{lemma}
\begin{proof}
  Let us consider two $6$--dimensional closed spin $e$-manifolds of the forms
  $\a= (W,\es,e)$ and $\a' = (W',\es,e')$  such that $[W,\es,e] = [W',\es,e']$ in
  $\O_{6}^{e,spin}$.
  We only need to show that the difference
  $\chi([W,\es,e]) - \chi([W',\es,e'])$ belongs to $16\Z \oplus 4\Z$.

  There exists a $7$--dimensional spin $e$-manifold $\b = (Z,X,\Te)$ such that $\bd{\b} = \a
  \amalg (-\a')$.
  The $4$--submanifold $X$ is closed, spin, and embedded in the interior of $Z$.
  Thus, the manifold $Z_{X}$  has the smooth spin boundary
  \begin{equation*}
	\bd{Z_{X}} = W \amalg (-W') \amalg (-\HX).
  \end{equation*}
  Since $\Te|_{\HX} = e(F_{X})$, we have
  \begin{equation*}
	\bd{(Z_{X},\Te|_{Z_{X}})} =
	(W,e) \amalg (-W',e') \amalg (-\HX,e(F_{X})),
  \end{equation*}
  and this implies $[W,e] - [W',e'] = [\HX,e(F_{X})]$ in $\Os_{6}(K(\Q,2))$.  
  By \fullref{lem:cu}, we have 
  \[
	\chi([\HX, e(F_{X})]) = \chi(\up([X,\nu_{X}])) = \xi([X,\nu_{X}]) \in 16\Z \oplus 4\Z,
  \]
  where $\nu_{X}$ is the normal bundle of $X$.
\end{proof}
Now, we can prove \fullref{main:O}.
\begin{proof}[Proof of \fullref{main:O}]
  Consider the following commutative diagram:
  \begin{equation*}
	\begin{CD}
	  0 @>>>
	  \Os_{4}(BSpin(3))
	  @>{\up}>>
	  \Os_{6}(K(\Q,2))
	  @>{\pi}>>
	  \Oes
	  @>>> 0\\
	  @.   @| @| @V{\Phi}VV @. \\
	  0 @>>>
	  \Os_{4}(BSpin(3))
	  @>{\up}>>
	  \Os_{6}(K(\Q,2))
	  @>{\chi'}>>
	  \frac{\Q \oplus \Q}{16\Z \oplus 4\Z}
	  @>>> 0\\
	\end{CD}
  \end{equation*}
  The lower horizontal sequence is exact by \fullref{prop:exact-O}, and
  the homomorphism $\pi$ is surjective by \fullref{prop:pi-onto}.
  To complete the proof, we only have to show that the upper horizontal sequence is exact,
  more specifically, 
  \[
	\Image \up = \Ker \pi.
  \]
  We prove this in two steps as follows.

  \textit{\hypertarget{claim:sub}{Claim~1}\textup{:} $\Image \up \subset \Ker \pi$.}\
  Let $[X,E] \in \O_{4}^{spin}(BSpin(3))$ be any element, then 
  \[
  	\pi(\up([X,E])) = [S(E),\es,e(F_{E})]
  \]
  by definition.
  We can regard $X$ as the image of the zero--section of $E$ so that
  $X \subset \Int D(E)$.
  The cohomology class $e(F_{E})$ is an $e$-class of $(S(E),\es) = \bd{(D(E),X)}$, and it 
  uniquely extends to an $e$-class, say $e_{E}$, of $(D(E),X)$.  The obtained
  spin $e$-manifold $(D(E),X,e_{E})$ 
  is bounded by $(S(E),\es,e(F_{E}))$, and hence, we have $\pi(\up([X,E])) = 0$.

  \textit{\hypertarget{claim:sup}{Claim~2}\textup{:} $\Image \up \supset \Ker \pi$.}
  Next, we prove the opposite inclusion.
  Let $[W,e] \in \Ker \pi$ be any element, then $\a = (W,\es,e)$ bounds a
  $7$--dimensional spin $e$-manifold $\b = (Z,X,\Te)$, namely $\bd{\b} = \a$.
  In particular, we have $\Te|_{\HX} = e(F_{X})$.
  Since 
  \[
	\bd{(Z_{X}, \Te|_{Z_{X}})} = (W,e) \amalg (-\HX, e(F_{X})),
  \]
  the cobordism class $[W,e] \in \O_{6}^{spin}(K(\Q,2))$ satisfies
  \begin{equation*}
	[W,e] = [\HX,e(F_{X})] = \up([X,\nu_{X}]) \in \Image \up,
  \end{equation*}
  where $\nu_{X}$ is the normal bundle of $X$.
\end{proof}

\section{\texorpdfstring{Signature modulo $32$}{Signature modulo 32}}
\label{sec:signature32}
Let $M$ be an oriented integral homology $3$--sphere, 
and $\a_{M}= (Y,Q,e_{M})$ the $6$--dimensional closed spin $e$-manifold constructed in
\fullref{sec:construction}.
Let $[\a_{M}] \in \O_{6}^{e,spin}$ denote the spin cobordism class of $\a_{M}$.
In this section, 
by using the isomorphism $\Phi$,
we derive a necessary and sufficient condition for the vanishing 
$[\a_{M}] = 0$ in terms of the signature of a  $4$--manifold
(\fullref{prop:W32}).

Recall that we constructed a $G$--action on $(Y,Q)$ in \fullref{sec:involution}.
The normal bundle $\nu_{Q}$ of $Q$ has a $G$--equivariant trivialization
$\nu_{Q} = Q \x \R^{3}$ such that
\begin{align}
  \i(x,v) &= (\i(x), -v),
  \label{eq:nuQ}\\
  \HQ &= Q \x S^{2},
  \label{eq:HQ}
\end{align}
where $(x,v) \in \nu_{Q}$.

Let $X_{0}$ be an oriented spin $4$--manifold equipped with an identification
$\bd{X_{0}} = M$, and consider the union
\begin{equation*}
  X = X_{1} \cup X_{2} \cup X_{3},
\end{equation*}
where $X_{i}$ ($i=1,2,3$) are disjoint copies of $X_{0}$ such that $\bd{X_{i}} = M_{i}$, and so
\begin{equation}
  \label{eq:XQ}
  \bd{X} = Q.
\end{equation}
The $G$--action on $Q$ naturally extends to an action on $X$ such that $\i(X_{1}) =
\i(X_{2})$ and that $\i$ restricts to the identity on $X_{3}$.
We define a $G$--action on the trivial vector bundle $X \x \R^{3}$ over $X$ in the same way as
\eqref{eq:nuQ}.
Then, the $G$--vector bundle $X \x \R^{3}$ restricts to $\nu_{Q}$ over $Q$.
Consequently,
\begin{equation}
  \bd{X \x S^{2}} = \HQ.
  \label{eq:XHQ}
\end{equation}
Note that \eqref{eq:HQ}, \eqref{eq:XQ}, and \eqref{eq:XHQ} hold
as $G$--manifolds.

As in \eqref{eq:Wp} and \eqref{eq:ep}, let us consider 
the closed spin $6$--dimensional $G$--manifold
\begin{equation}
  Y' = Y_{Q} \cup_{\HQ} (- X \x S^{2})
  \label{eq:Yp}
\end{equation}
obtained by gluing the common boundaries $\bd{Y_{Q}} = \HQ = \bd{(X \x S^{2})}$,
and the cohomology class $e_{M}' \in H^{2}(Y';\Q)$ such that
\begin{equation}
  e_{M}'|_{Y_{Q}} = e_{M},\quad e_{M}'|_{X \x S^{2}} = f_{X}^*e(TS^{2}),
  \label{eq:e}
\end{equation}
  where
\begin{equation}
  f_{X} \co X \x S^{2} \to S^{2}
  \label{eq:fX}
\end{equation}
is the projection.
We obtain a $6$--dimensional closed spin $e$-manifold
\begin{equation*}
  \a_{M}' = (Y',\es,e_{M}').
\end{equation*}
Note that the $G$--action on $Y'$ is free, and the quotient $Y'/G$ is a smooth closed
unoriented manifold.

\begin{lemma}
  \label{lem:injective}
  For $k \leq 3$, the restriction homomorphisms
  \begin{align*}
	H^{k}(Y';\Z) &\to H^{k}(X \x S^{2};\Z),\\
	H^{k}(Y'/G;\Z/2) &\to H^{k}((X \x S^{2})/G;\Z/2)
  \end{align*}
  are injective.
\end{lemma}
\begin{proof}
  We identify the cohomology group
  $H^{*}(Y',X\x S^{2};\Z)$ with $H^{*}(Y_{Q}, \HQ;\Z)$,
  and $H^{*}(Y'/G, (X \x S^{2})/G;\Z/2)$ with $H^{*}(Y_{Q}/G, \HQ/G;\Z/2)$
  by using the excision isomorphisms. 

  The homomorphism $\d^* \co H^{k-1}(X \x S^{2};\Z) \to H^{k}(Y', X \x S^{2};\Z)$ given by the pair $(Y', X \x
  S^{2})$ coincides with the composition of two homomorphisms
  \begin{equation}
	\label{eq:composition}
	H^{k-1}(X;\Z) \to H^{k-1}(\HQ;\Z) \to H^{k}(Y_{Q},\HQ;\Z),
  \end{equation}
  where the first arrow is the restriction,
  and where the second arrow is the homomorphism given by $(Y_{Q},\HQ)$.
  Note that the homomorphism $H^{k}(Y_{Q},\HQ;\Z) \to H^{k}(Y_{Q};\Z)$ is trivial.
  Both homomorphisms in \eqref{eq:composition} are surjective,
  and so is $\d^*$.
  Hence, $H^{k}(Y';\Z) \to H^{k}(X \x S^{2};\Z)$ is injective.

  Similarly, the homomorphism $H^{k-1}(X/G;\Z/2) \to H^{k}(Y'/G, (X \x S^{2})/G;\Z/2)$ coincides
  with the composition of two surjective homomorphisms
  \begin{equation*}
	H^{k-1}(X/G;\Z/2) \to H^{k-1}(\HQ/G;\Z/2) \to H^{k}(Y_{Q}/G,\HQ/G;\Z/2).
  \end{equation*}
  Note that the homomorphism $H^{k}(Y_{Q}/G, \HQ/G;\Z/2) \to H^{k}(Y_{Q}/G;\Z/2)$ is trivial. 
  Therefore,
  $H^{k}(Y'/G;\Z/2) \to H^{k}(Y'/G, (X \x S^{2})/G; \Z/2)$ is injective.
\end{proof}
The following lemma is easy to prove.
\begin{lemma}
  $e_{M}' \equiv 0 \pmod{2}$.
  \label{lem:mod4}
\end{lemma}
\begin{proof}
  The Euler characteristic of $S^{2}$ is $2$, which is even.
  Therefore, the cohomology class $e_{M}' \mod{2}$ belongs to the kernel 
  of the restriction
  \[
	 H^{2}(Y';\Q/2\Z) \to H^{2}(X \x S^{2};\Q/2\Z)
  \]
  by \eqref{eq:e}.
  On the other hand, this homomorphism is injective by \fullref{lem:injective}.
  Therefore, $e_{M}' \equiv 0 \pmod{2}$.
\end{proof}
\begin{proposition}
  \label{prop:W32}
  Assume that there is a $4$--submanifold $W$ of $Y'$ which Poincar\'e dual is $e_{M}'$.
  Then, the $e$-manifold $\a_{M}$ is spin null--cobordant
  \textup{(}namely, \textup{\fullref{main:spin-cob}~\eqref{item:spin-null}} holds\textup{)}
  if, and only if,  
  \[
	\Sign W \equiv 0 \pmod{32}.
  \]
\end{proposition}
\begin{proof}
  By \fullref{prop:pi-onto}, $[\a_{M}] = [\a_{M}']$ in $\O_{6}^{e,spin}$.
  By \eqref{eq:Phi2},  we have
  \begin{equation*}
	\Phi([\a_{M}']) \equiv
	\left( 
	\frac{1}{6}\int_{Y'}^{} p_{1}(TY') e_{M}' - e_{M}'^{3},\ 
	\frac{1}{2}\int_{Y'}^{}e_{M}'^{3}
	\right) \mod{16\Z \oplus 4\Z}.
  \end{equation*}
  Since $p_{1}(TY')|_{W} = p_{1}(TW) + e_{M}'^{2}$, the first component on the right--hand side 
  is equal to $\Sign W/2$.
  By \fullref{lem:mod4}, we have $e_{M}'^{3}/2 \equiv 0 \pmod{4}$, and
  so $\Phi([\a_{M}]) \equiv (\Sign W/2, 0)\mod{16\Z \oplus 4\Z}$.
  Since $\Phi$ is an injective by \fullref{main:O}, 
  $[\a_{M}] = 0$ if, and only if, $\Sign W/2 \equiv 0 \pmod{16}$. 
\end{proof}

\section{\texorpdfstring{Proof of \fullref{main:spin-cob}}{Proof of Theorem \ref{main:spin-cob}}}
\label{sec:congruence}
In this section, we prove \fullref{main:spin-cob}~\eqref{item:spin-null}, by
constructing a $4$--submanifold $W$ of $Y'$ as in \fullref{prop:W32}.
\begin{proposition}
  \label{prop:F}
  There exists an oriented vector bundle $F$ of rank $2$ over $Y'$ with a $G$--action satisfying the
  following two properties.
\begin{enumerate}
  \item\label{item:FQ} $e(F) = e_{M}'$ over $\Q$.
  \item\label{item:w} $w_{i}(F/G) = w_{i}(TY'/G)$ in $H^{i}(Y'/G;\Z/2)$ for $i=1,2$.
\end{enumerate}
\end{proposition}
Here, $F/G$ is the quotient of $F$, which is unoriented vector bundle of rank $2$ over
the unoriented manifold $Y'/G$, and here,
$w_{i}$ denotes the $i$--th Stiefel--Whitney class.
The proof will be given in \fullref{sec:vector-bundle}.

Since $G$ acts freely on $Y'$, 
there exits a $G$--equivariant smooth section $s \co Y' \to
F$ transverse to the zero section.
We define
$W = \left\{ x \in Y' \ | \ s(x)= 0 \right\}$,
which is a smooth oriented closed $4$--dimensional $G$--submanifold of $Y'$.
By \fullref{prop:F}~\eqref{item:FQ},
the Poincar\'e dual of $W$ is $e_{M}'$.
The quotient space  $W/G$ is a unoriented smooth submanifold of $Y'/G$.
\begin{lemma}
  $w_{i}(TW/G) = 0$ for $i=1,2$.
  \label{lem:wTW}
\end{lemma}
\begin{proof}
  There is an isomorphism
  $ TY'/G|_{W/G} \cong TW/G \oplus F/G|_{W/G}$.
  Since $TY'/G$ and $F/G$ have the same Stiefel--Whitney classes $w_{i}$
  ($i=1,2$) by \fullref{prop:F}~\eqref{item:w}, we have
  $w_{i}(TW/G) = 0$.
\end{proof}
We can prove \fullref{main:spin-cob}~\eqref{item:spin-null} as follows.
\begin{proof}[Proof of \fullref{main:spin-cob}~\eqref{item:spin-null}]
  By \fullref{lem:wTW}, the closed smooth manifold $W/G$ is orientable and spinnable.
  We fix an orientation of $W/G$, then $\Sign W = \pm 2 \Sign W/G$.
  By Rokhlin's theorem \eqref{eq:rokhlin}, we have $\Sign W/G \equiv 0 \pmod{16}$, and 
  consequently, $\Sign  W \equiv 0 \pmod{32}$.
\end{proof}

\section{\texorpdfstring{$G$--vector bundle}{G-vector bundle}}
\label{sec:vector-bundle}
In this section, we prove \fullref{prop:F}.
To construct the $G$--vector bundle $F$,
we prove the existence of a $G$--equivariant classifying map $f_{M} \co Y' \to \CP^{3}$ of $F$.
Here, a $G$--action on $\CP^{3}$ is defined as follows.

Let $\H$ denote the quaternions spanned by $\left\{ 1, i, j, k \right\}$ over $\R$ such that
$i^{2} = j^{2} = k^{2} = i j k = -1$.
By regarding $\H$ as the complex space $\C \oplus \C j$,
we can identify the complex projective space $\CP(\H^{n+1})$ with $\CP^{2n+1}$ for $n \geq 0$ 
(our main interest is when $n=0, 1$).
The multiplication by $j$ on vectors on $\H^{n+1}$ from the left provides a free involution $\i
\co \CP^{2n+1} \to \CP^{2n+1}$, and so $\CP^{2n+1}$ is a $G$--manifold.
Note that the natural inclusion $S^{2} =\CP^{1} \to \CP^{3}$ commutes with the $G$--action.
The unit $2$--sphere $S^{2} \subset \R^{3}$ has a free $G$--action given by the multiplication
by a scalar $-1$.
We shall identify $\CP^{1}$ with $S^{2}$ as a $G$--manifold.

Let $f_{Q} \co \HQ \to S^{2}$ be the projection onto the fiber given by the
trivialization \eqref{eq:HQ}, and
$f_{X} \co X \times S^{2} \to S^{2}$  be as in \eqref{eq:fX}.  
Let $S^{3}_{i}$ ($i=1,2$) be as in  \eqref{eq:bouquet}.
Note that $f_{X}|_{\HQ} = f_{Q}$ and $\i(S^{3}_{1}) = S^{3}_{2}$.

Let $P_{i}$ ($i=1,2,3$) be $0$--dimensional submanifolds of $Y$ defined as follows:
\begin{eqnarray*}
  P_{1} = \left\{ (0,\infty) \right\},\qquad
  P_{2} = \left\{ (\infty, 0) \right\},\qquad
  P_{3} = \left\{ (\infty, \infty) \right\},
\end{eqnarray*}
then
$S^{3}_{1} \cap Q = P_{1} \cup (-P_{3})$ and $S^{3}_{2} \cap Q = P_{2} \cup (-P_{3})$
as oriented manifolds.
We define
\begin{equation*}
  C_{i} = S^{3}_{i} \setminus (Q \x \Int D^{3}) \quad \text{($i=1,2$)},
\end{equation*}
which is a proper $3$--submanifold of $Y_{Q}$.
We shall assume that $C_{i}$ is diffeomorphic to $S^{2}\times [0,1]$, by choosing a small tubular
neighborhood $Q \x D^{3}$ of $Q$ (so that $S^{3}_{i} \cap (Q \x D^{3})$ is the disjoint union of two small
$3$--balls in $S^{3}_{i}$).
In particular, the boundary $\bd{C_{i}}$ is the disjoint union two $2$--spheres
$\bd{C_{i}} = (S^{3}_{i} \cap \HM_{i}) \amalg (-S^{3}_{i} \cap \HM_{3})$.
The involution $\i \co Y' \to Y'$ restricts to a diffeomorphism $\i|_{C_{1}} \co C_{1} \to C_{2}$.
We write
\begin{equation*}
  C = \HQ \cup C_{1} \cup C_{2},
\end{equation*}
then $\i(C) = C$.
\begin{lemma}
  The map $f_{Q} \co \HQ \to S^{2}$ extends to a $G$--equivariant map 
  $f_{C} \co C \to S^{2}$.
  \label{lem:fC}
\end{lemma}
\begin{proof}
  By the definition of $f_{Q}$, the two maps
  \begin{equation*}
	f_{Q}|_{S^{3}_{1} \cap \HM_{1}} \co S^{3}_{1} \cap \HM_{1} \to S^{2},\quad
	f_{Q}|_{S^{3}_{1} \cap \HM_{3}} \co S^{3}_{1} \cap \HM_{3} \to S^{2}
  \end{equation*}
  have the degree $+1$ and $-1$ respectively.
  Therefore $f_{Q}|_{\bd{C_{1}}} \co \bd{C_{1}} \to S^{2}$ extends to a map $f_{C_{1}} \co C_{1}
  \to S^{2}$. 
  We define a map $f_{C} \co C \to S^{2}$ by
  \begin{equation*}
	f_{C}(x) = 
	\begin{cases}
	  f_{Q}(x) & \text{if $x \in \HQ$}\\
	  f_{C_1}(x) & \text{if $x \in C_{1}$}\\
	  f_{C_{1}}(\i(x)) & \text{if $x \in C_{2}$}
	\end{cases}
  \end{equation*}
  for $x \in C$.
  It is easy to check that this is well--defined and $G$--equivariant.
\end{proof}

To obtain a classifying map $f_{M} \co Y_{Q} \to \CP^{3}$, 
we consider the obstruction classes to extending the map $f_{C}$ to
a $G$--equivariant map $f_{M}$.   
The primary obstruction class belongs to the cohomology group
\begin{equation}
  \label{eq:obstruction}
  H^{3}(Y_{Q}/G, C/G; \Z_{-}),
\end{equation}
where $\Z_{-}$ is the local system on $Y_{Q}/G$
given by the non--trivial 
characteristic homomorphism $G \to \Aut(\pi_{2}(\CP^{3})) = \left\{ id, -id \right\}$.
In other words, \eqref{eq:obstruction} is the $G$--equivariant cohomology group with
coefficients in the non--trivial $G$--module $\Z$
(such that $\i 1 = -1$).
\begin{lemma}
  \label{lem:o}
  The obstruction group \eqref{eq:obstruction} vanishes.
\end{lemma}
\begin{proof}
  The low--dimensional cohomology groups of $(Y_{Q}, C)$ and $(Y_{Q}/G, C/G)$ are given as follows:
  \begin{equation}
	\label{eq:H-vanish}
	  H^{k}(Y_{Q}, C; \Z) \cong
	  H^{k}(Y_{Q}/G, C/G; \Z) \cong
	  0
	  \quad \text{($k\leq 3$)}
  \end{equation}
  There is a long exact sequence 
  \begin{equation*}
	\dotsb \xrightarrow{\d^*}
	H^{k}(Y_{Q}/G, C/G; \Z_{-})
	\xrightarrow{q^{*}}
	H^{k}(Y_{Q}, C;\Z)
	\xrightarrow{q_{!}}
	H^{k}(Y_{Q}/G, C/G; \Z)
	\to \dotsb,
  \end{equation*}
  where $q^*$ is the pull--back of the covering map $q \co Y_{Q} \to Y_{Q}/G$, and
  $q_{!}$ is the Gysin homomorphism.
  The vanishing \eqref{eq:H-vanish} and the exact sequence implies 
  \[
  H^{k}(Y_{Q}/G, C/G; \Z_{-}) = 0\quad \text{($k \leq 3$)}.
  \proved
  \]
\end{proof}
\begin{proposition}
  There exists a $G$--equivariant map $f_{M} \co Y' \to \CP^{3}$ such that $f_{M}|_{\HQ}
  =f_{Q}$.
  \label{lem:map}
\end{proposition}
\begin{proof}
  By \fullref{lem:o}, the primary obstruction class vanishes.
  The higher obstruction  groups  vanishes, since
  $\pi_{i}(\CP^{3}) = 0$ for $3 \leq i \leq 6$.
  Hence, $f_{C}$ extends to a $G$--equivariant map $f_{M} \co Y_{Q} \to \CP^{3}$.
\end{proof}

Now,  let us consider the fiber bundle
$
  \rho \co \CP^{3} \to \HP^{1}
$
which maps a complex line $l$ in $\H^{2}$ to the corresponding quaternionic
line $\H \otimes_{\C}l$.
The fiber of $\rho$ is $\CP^{1}$, and the $G$--action preserves the fiber.
Let $F_{1} \subset T\CP^{3}$ be the tangent subbundle of $\CP^{3}$ with respect to $\rho$,
which is an oriented vector bundle of rank $2$ over $\CP^{3}$ with a $G$--action.
We then define
\begin{equation}
  \label{eq:F}
  F = f_{M}^*F_{1}
\end{equation}
to be the pull--back of $F_{1}$ under $f_{M}$.
It is an oriented vector bundle of rank $2$ over $Y'$ with a $G$--action.

\begin{proof}[Proof of \fullref{prop:F}]
  By the construction of $F$, we have
  \[
	e(F)|_{X \x S^{2}} = f_{X}^*e(TS^{2}) = e_{M}'|_{X \x S^{2}}.
  \]
  By \fullref{lem:injective}, the homomorphism $H^{2}(Y';\Q) \to H^{2}(X \x S^{2};\Q)$ is
  injective, and therefore, $e(F) = e_{M}'$, and \fullref{prop:F}~\eqref{item:FQ} holds.

  The vector bundle $F$ restricts to $f_{X}^* TS^{2}$ over $X \x S^{2}$.
  The quotient manifold $(X \x S^{2})/G$ is diffeomorphic to the disjoint union of $X \x S^{2}$
  and $X \x \RP^{2}$. 
  Since $X$ is oriented and spin, we have
  \begin{equation*}
	w_{i}(TY'/G)|_{(X \x S^{2})/G} = w_{i}(f_{X}^*TS^{2}/G) 
	= w_{i}(F/G)|_{(X \x S^{2})/G}\quad \text{($i=1,2$)}.
  \end{equation*}
  By \fullref{lem:injective}, we have $w_{i}(TY'/G) = w_{i}(F/G)$.
  Namely, \fullref{prop:F}~\eqref{item:w} holds.
\end{proof}

\section{\texorpdfstring{Appendix: Yet another proof of \fullref{maincor}}{Appendix: Yet
another proof of Corollary \ref{mincor}}}
\label{sec:yap}

Let $M$ be an oriented integral homology $3$--sphere, and $\a_{M} = (Y,Q,e_{M})$ the
$6$--dimensional closed spin $e$-manifold constructed from $M$.
The aim of this section is to give
yet another direct proof of \fullref{maincor} using
\fullref{main:spin-cob} and without using \fullref{thm:sigma}.

\begin{proof}
  Let us assume $M \cong -M$, then $\a_{M} \cong -\a_{M}$ by
  \fullref{main:spin-cob}~\eqref{item:reverse}.
Namely, there exists a diffeomorphism 
\begin{equation*}
  h \co (Y,Q) \to (Y,Q)
\end{equation*}
which reverses the orientations of $Y$ and $Q$ such that $h^*e_{M} = e_{M}$.
By \fullref{main:spin-cob}~\eqref{item:spin-null},
there exists a $7$--dimensional spin $e$-manifold
$
  \b = (Z,X,\Te)
$
such that $\bd{\b} = \a_{M}$.

Let us consider the $7$--dimensional closed spin $e$-manifold
\begin{equation*}
  \b' = \b \cup_{h} \b
\end{equation*}
obtained by gluing the boundaries of two disjoint copies of
$\b$ by using $h$, more precisely,  we can write
\begin{equation*}
  \b' = (Z',X',\Te'),\quad
  Z' = Z \cup_{h}Z,\quad
  X' = X \cup_{h}X,
\end{equation*}
where $\Te' \in H^{2}(Z' \setminus X'; \Q)$ is the $e$-class of $(Z',X')$
obtained by gluing two copies of $\Te$.
Note that the manifolds $Z'$ and $X'$ are closed spin.

What we need to prove is $\Sign X \equiv 0 \pmod{16}$, or equivalently
\begin{equation*}
  \Sign X' \equiv 0 \pmod{32}.
\end{equation*}
It is easy to show that there is the following formula (see also \fullref{lem:cu}):
\begin{equation*}
  \Sign X'  = \frac{1}{6} \int_{\HX'}^{}p_{1}(T\HX') e(F_{X'}) - e(F_{X'})^{3}
\end{equation*}
Since $\bd{Z'_{X'}} = -\HX'$ and $\Te'$ is an $e$-class of $(Z',X')$ 
(namely, $\Te'|_{\HX'} = e(F_{X'})$ by definition),
the right--hand side is equal to
\begin{equation*}
  -\frac{1}{6}\int_{\bd{Z'_{X'}}}^{} p_{1}(TZ'_{X'}) \Te' - \Te'^{3} = 0
\end{equation*}
by Stokes' theorem.
Consequently,
$ \Sign X' \equiv 0 \pmod{32}$.
\end{proof}





%
%
%
\bibliographystyle{gtart}
\bibliography{articles,books,preprints}
%

\end{document}